\newif\ifpictures
\picturestrue

\newif\iffullversion
\fullversiontrue

\documentclass[12pt]{amsart}
\usepackage{amssymb}
\usepackage[dvips]{graphicx}

\ifpictures
\usepackage{psfrag}
\fi

\numberwithin{equation}{section}
\newtheorem{thm}{Theorem}
\newtheorem{prop}[thm]{Proposition}
\newtheorem{lemma}[thm]{Lemma}
\newtheorem{cor}[thm]{Corollary}

\newtheorem{fact}[thm]{Fact}

\theoremstyle{definition}

\newtheorem{example}[thm]{Example}
\newtheorem{remark1}[thm]{Remark}
\newtheorem{openproblem1}[thm]{Open problem}

\newenvironment{ex}{\begin{example}\rm}{\end{example}}
\newenvironment{rem}{\begin{remark1}\rm}{\end{remark1}}

\numberwithin{thm}{section}

\newcounter{FNC}[page]
\def\newfootnote#1{{\addtocounter{FNC}{2}$^\fnsymbol{FNC}$%
     \let\thefootnote\relax\footnotetext{$^\fnsymbol{FNC}$#1}}}

\newcommand{\C}{\mathbb{C}}
\newcommand{\N}{\mathbb{N}}

\newcommand{\R}{\mathbb{R}}

\renewcommand{\P}{\mathbb{P}}

\newcommand{\Z}{\mathbb{Z}}

\newcommand{\T}{\mathbb{T}}

\newcommand{\Sph}{\mathbb{S}}
\newcommand{\V}{\mathcal{V}}
\newcommand{\coamoeba}{\mathcal{C}}
\newcommand{\cA}{\mathcal{A}}
\newcommand{\cU}{\mathcal{U}}
\newcommand{\cC}{\mathcal{C}}

\DeclareMathOperator{\conv}{conv}
\DeclareMathOperator{\supp}{supp}

\DeclareMathOperator{\trop}{trop}

\DeclareMathOperator{\Arg}{Arg}
\DeclareMathOperator{\Log}{Log}
\DeclareMathOperator{\New}{New}
\DeclareMathOperator{\Res}{Res} 
\DeclareMathOperator{\tdeg}{tdeg} 
\DeclareMathOperator{\vol}{vol}
\DeclareMathOperator{\re}{re}
\DeclareMathOperator{\im}{im}

\usepackage{xcolor}

\title{Approximating amoebas and coamoebas \\
by sums of squares} 

\author{Thorsten Theobald and Timo de Wolff}

\address{Goethe-Universit\"at, FB 12 -- Institut f\"ur Mathematik,
Postfach 11 19 32, D--60054 Frankfurt am Main, Germany}

\email{\{theobald,wolff\}@math.uni-frankfurt.de}

\thanks{Research supported by DFG grant TH 1333/2-1 and (first author) by the 
A.\ v.\ Humboldt-Foundation.}

\begin{document}

\begin{abstract}
Amoebas and coamoebas are the logarithmic images of algebraic varieties
and the images of algebraic varieties under the arg-map, respectively.
We present new techniques for computational problems on amoebas
and coamoebas, thus establishing new connections between (co-)amoebas, 
semialgebraic and convex algebraic geometry and semidefinite programming.

Our approach is based on formulating the membership problem in amoebas
(respectively coamoebas) as a suitable real algebraic feasibility problem.
Using the real Nullstellensatz, this allows to tackle the problem
by sums of squares techniques and semidefinite programming.
Our method yields polynomial identities as certificates of
non-containment of a point in an amoeba or coamoeba.
As the main theoretical result, we establish some degree bounds on the polynomial 
certificates. Moreover, we provide some actual computations of amoebas based 
on the sums of squares approach.
\end{abstract}

\maketitle

\markright{\uppercase{Approximating amoebas and coamoebas}}



\section{Introduction}
For an ideal $I \subset \C[Z_1, \ldots, Z_n]$ an amoeba
(introduced by Gel$'$fand, Kapranov, and Zelevinsky \cite{gkz-94},
see also the surveys \cite{mikhalkin-amoebastrop} or 
\cite{passare-tsikh-2005})
is the image of the variety $\V(I)$ where each complex coordinate
is mapped to (the logarithm of) its absolute value. 
It is often customary and useful to 
consider the logarithmic version of an \emph{amoeba} 
\[
  \mathcal{A}_I \ = \ \{ (\log |z_1|, \ldots, \log |z_n|) \, : \, z \in 
  \V(I) \cap (\C^*)^n \} 
\]
with $\C^* := \C \setminus \{0\}$, and we denote the \emph{unlog amoeba} by
\[
  \mathcal{U}_I \ = \ \{ (|z_1|, \ldots, |z_n|) \, : \, z \in \V(I) \} \, .
\]
Similarly, the \emph{coamoeba} $\mathcal{C}_I$ is defined as
$\mathcal{C}_I := \Arg(\V(I) \cap (\C^*)^n)$, where $\Arg$ denotes the
mapping
\[
  (z_1, \ldots, z_n) \mapsto (\arg (z_1), \ldots, \arg(z_n)) \in
 \T^n := (\R / 2 \pi \Z)^n
\]
and $\arg$ denotes the argument of a complex number
(see \cite{nilsson-diss,nilsson-passare-2009,nilsson-passare-2010,
nisse-2009-geometriccombinatorial}).
If $I$ is a principal ideal generated by a polynomial $f$,
we shortly write $\mathcal{A}_f := \mathcal{A}_{\langle f \rangle}$ and
analogously $\mathcal{U}_f$, $\mathcal{C}_f$.

Studying computational questions of amoebas has been initiated
in \cite{theobald-expmath-2002}, where in particular certain
special classes of amoebas (e.g., two-dimensional amoebas,
amoebas of Grassmannians) were studied. One of the natural and
fundamental computational questions is the membership problem, 
which asks for a given point $(\lambda_1, \ldots, \lambda_n)$ whether
this point is contained in an (unlog) amoeba respectively coamoeba.

In \cite{purbhoo-2008} Purbhoo provided a characterization for the
points in the complement of a hypersurface amoeba which can be used
to numerically approximate the amoeba. His \emph{lopsidedness criterion}
provides an inequality-based certificate for non-containment of a point 
in an amoeba, but does not provide an algebraic certificate
(in the sense of a polynomial identity certifying the non-containment).
The certificates are given by iterated resultants.
With this technique the amoeba can be approximated by a limit process.
The computational efforts of computing the resultants are growing quite
fast, and the convergence is slow.

A different approach to tackle computational problems on amoebas
is to apply suitable Nullstellen- or
Positivstellens\"atze from real algebraic geometry
or complex geometry. For some natural problems a 
direct approach via the Nullstellensatz (applied on a 
realification of the problem)
is possible. Using a degree truncation approach, this allows
to find sum-of-squares-based polynomial identities which
certify that a certain point is located outside of an amoeba
or coamoeba. In particular, it is well known from recent
lines of research in computational semialgebraic 
geometry (see, e.g., \cite{lasserre-2001,laurent-2009-survey,parrilo-2003})
that
these certificates can be computed via semidefinite programming (SDP).

In this paper, we discuss theoretical foundations as well as
some practical issues of such an approach, thus establishing
new connections between amoebas, semialgebraic and convex algebraic
geometry and semidefinite programming.
Firstly, we present various Nullstellensatz-type formulations
(a \textit{standard approach} in Statement~\ref{co:amoebacertificate} and a \textit{monomial approach} in Statement \ref{co:monomialbasedcertificate}) and compare their properties to a recent toric Nullstellensatz of
Niculescu and Putinar \cite{niculescu-putinar-2010}.
Using a degree truncation approach this yields a sequence of supersets of the amoeba $\cA_f$, which converges to $\cA_f$ (Theorem~\ref{th:boundeddegree}). For every fixed superset in this sequence, the membership problem can be solved by semidefinite programming.

The main theoretical contribution is contained in 
Section~\ref{se:specialcertificates}.
For one of our approaches, we can provide some degree bounds for
the certificates (Corollary~\ref{co:certificate}). It is remarkable
and even somewhat surprising
that these degree bounds are derived from Purbhoo's lopsidedness criterion
(which is not at all sum-of-squares-based).
We also show that
in certain cases (such as for the Grassmannian of lines) the degree 
bounds can be reduced to simpler amoebas (Theorem~\ref{th:transfercertificates}).

In Section~\ref{se:computing} we provide some actual
computations on this symbolic-numerical approach. Besides providing
some results on the membership problem itself, we will also consider
more sophisticated versions (such as bounding the diameter of a complement
component for certain classes). 

Finally, in Section~\ref{se:outlook}
we give an outlook on further questions on the approach initiated
in the current paper.

\section{Preliminaries}

In the following, let $\C[Z] = \C[Z_1, \ldots, Z_n]$
denote the polynomial ring over $\C$ in $n$ variables.
For $f = \sum_{\alpha \in A} c_{\alpha} Z^{\alpha} \in \C[Z]$,
the Newton polytope 
$\New(f) = \conv \{ \alpha \in \N_0^n \, : \, \alpha \in \supp(f) \}$
of $f$ is the convex hull of the exponent vectors,
where $\supp(f)$  denotes the support of $f$.

\begin{figure}
\includegraphics[width=0.45\linewidth]{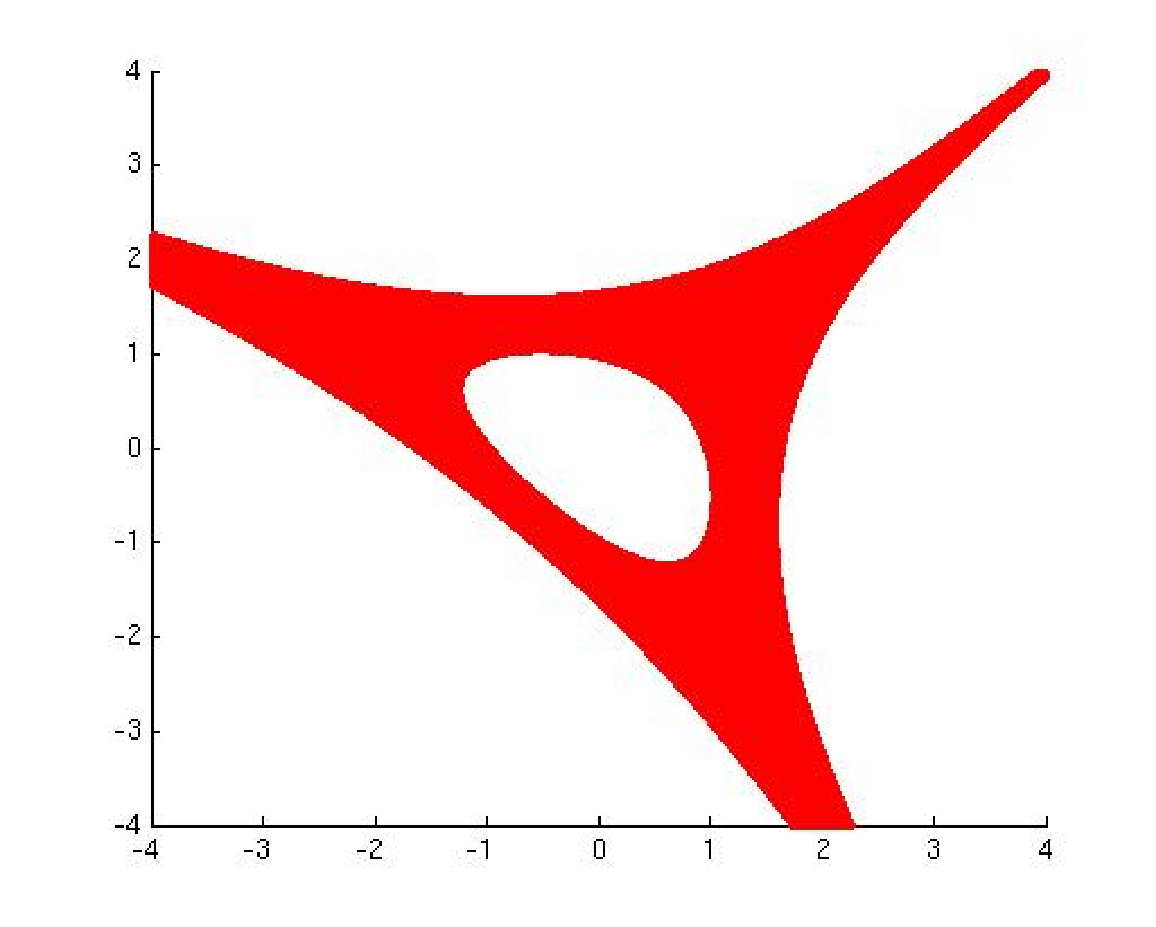}
\caption{The amoeba of $f = Z_1^2Z_2 + Z_1Z_2^2 - 4Z_1Z_2 + 1$.}
\label{FigureExampleAmoeba}
\end{figure}

\subsection{Amoebas and coamoebas}
We recall some basic statements about amoebas 
(see \cite{fpt-2000,mikhalkin-amoebastrop,passare-tsikh-2005}).
For any ideal $I \subset \C[Z]$, the
amoeba $\mathcal{A}_I$ is a closed set.
For a polynomial $f \in \C[Z]$, 
the complement of the hypersurface amoeba $\mathcal{A}_f$
consists of finitely many convex regions, and
these regions are in bijective correspondence with the different
Laurent expansions of the rational function $1/f$.
See Figure~\ref{FigureExampleAmoeba} for an example.

The order $\nu$ of a point $w$ in the complement 
$\R^n \setminus \mathcal{A}_f$ is given by
\[
  \nu_j \ = \ \frac{1}{(2\pi i)^n} \int_{\Log^{-1}(w)} \frac{z_j \partial_j f(z)}{f(z)}
    \frac{dz_1 \cdots dz_n}{z_1 \cdots z_n} \, , \; 1 \le j \le n \, ,
\]
where $\Log(z)$ is defined as $\Log(z) = (\log |z_1|, \ldots, \log |z_n|)$.
The order mapping induces an injective map from the set of complement
components into $\New(f) \cap \Z^n$.
The complement components corresponding to the vertices of $\New(f)$ do always exist \cite{fpt-2000}.

Similarly, for $f \in \C[Z]$
any connected component of the coamoeba complement
$\T^n \setminus \coamoeba_f$ is a convex set (see Figure~\ref{FigureExampleCoamoeba}).
If $\overline{\coamoeba_f}$ denotes the closure of $\coamoeba_f$
in the torus $\T^n$, then the number of connected components of
$\T^n \setminus \overline{\coamoeba_f}$ is bounded by
$n! \vol \New(f)$
(Nisse \cite[Thm. 5.19]{nisse-2009-geometriccombinatorial}),
where $\vol$ denotes the volume.

\begin{figure}
  \includegraphics[width=0.45\linewidth]{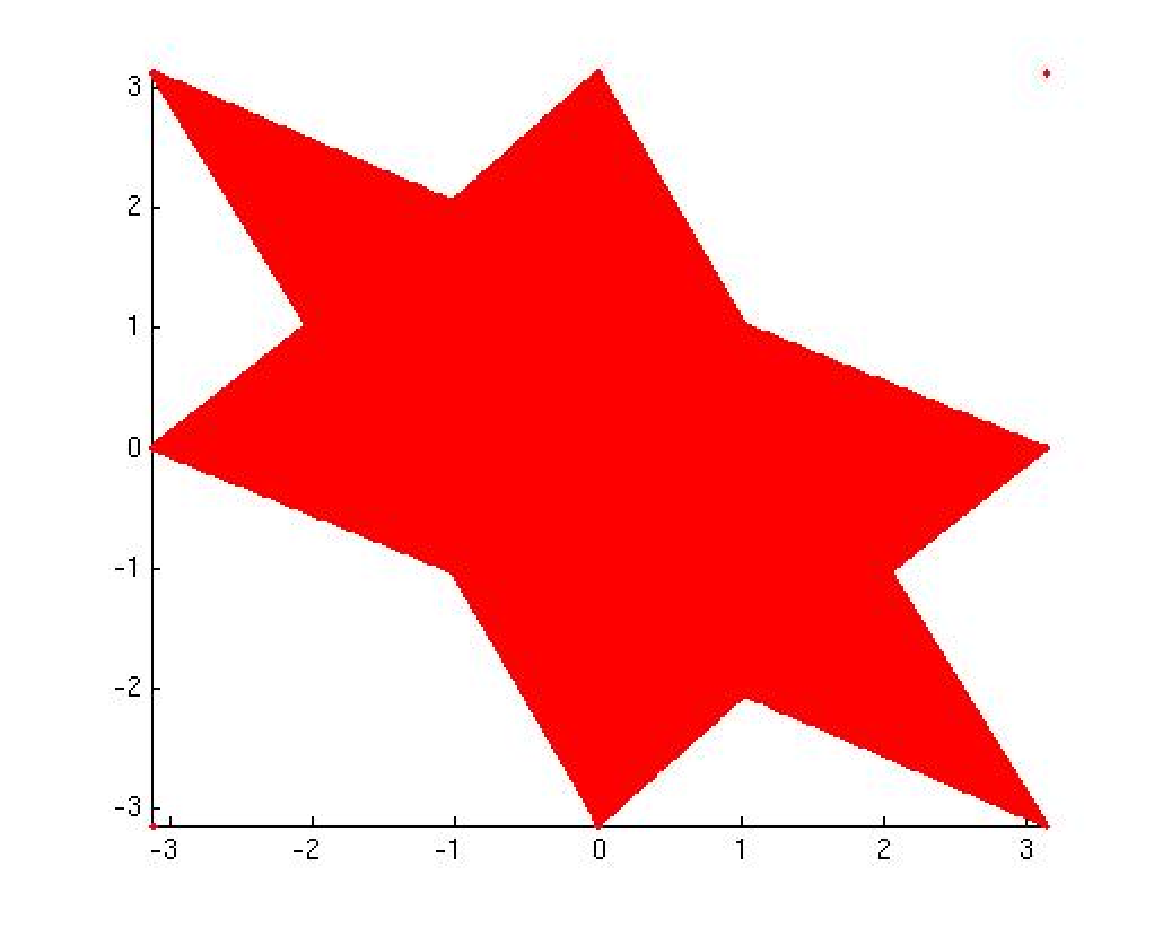} 
  \includegraphics[width=0.45\linewidth]{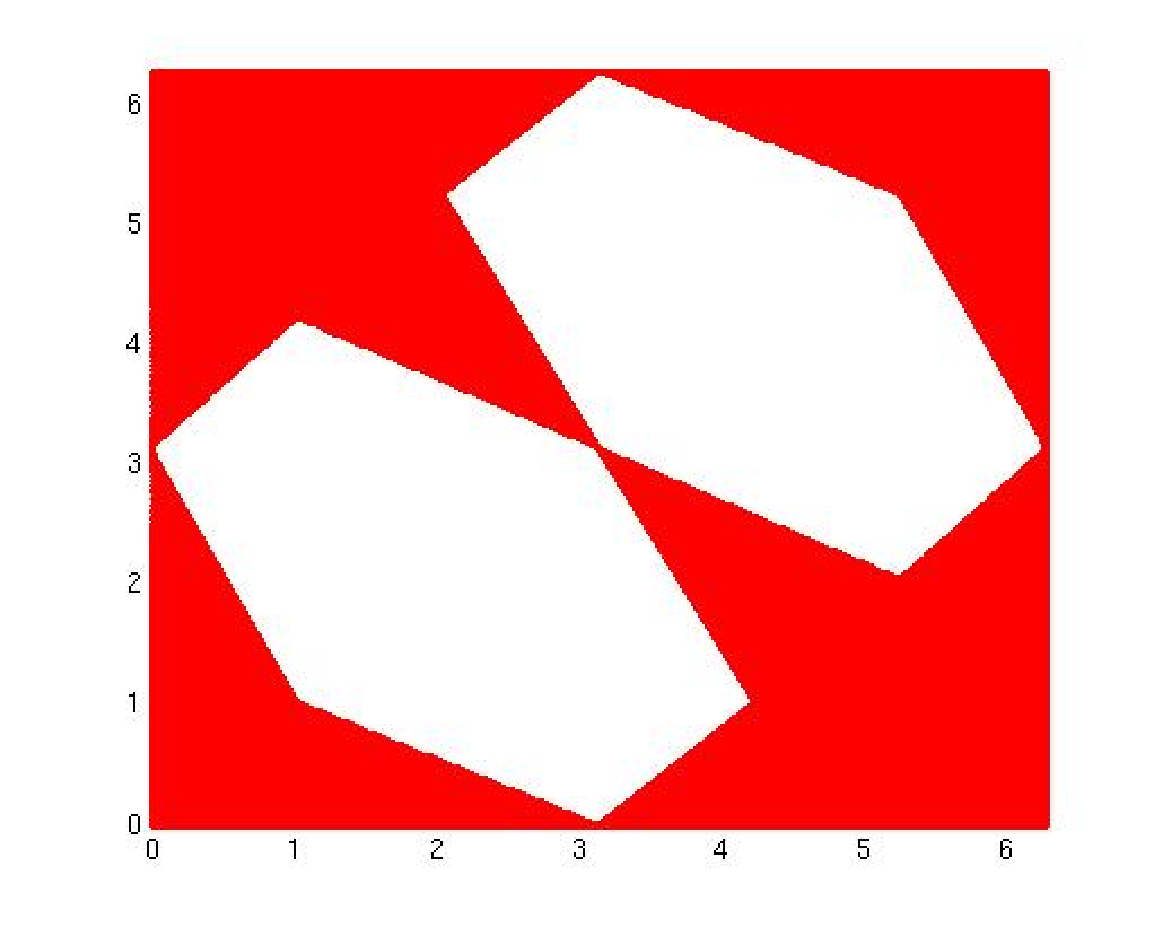}
\caption{The coamoeba of $f = Z_1^2Z_2 + Z_1Z_2^2 - 4Z_1Z_2 + 1$ in two different views
of $\T^2$, namely $[-\pi,\pi)^2$ versus $[0,2\pi)^2$.}
\label{FigureExampleCoamoeba}
\end{figure}

For technical reasons (see Theorem \ref{Thm:CoamoebaCertificate}) it will be often
convenient to consider in the definition of a coamoeba also
those points $z \in \V(I)$ which have a zero-component.
Namely, if a zero $z$ of $I$ has a zero-component $z_j = 0$
then we associate this component to \emph{any} phase. Call this
modified version of a coamoeba $\mathcal{C}_I'$, i.e., 
\begin{eqnarray*}
	\mathcal{C}_I' & :=	& \{\phi \in \T^n \ : \ \exists z \in \V(I): \arg(z_j) = \phi_j \text{ or } z_j = 0 \text{ for } 1 \leq j \leq n\}.
\end{eqnarray*}

Note that for principal ideals $I=\langle f \rangle$
the difference between $\mathcal{C}_I$ and $\mathcal{C}_I'$ 
solely may occur at points which are contained in the closure of
$\mathcal{C}_I$.
The set-theoretic difference of $\mathcal{C}_I$ and $\mathcal{C}_I'$
is a lower-dimensional subset of $\R^n$ 
(since in each environment of a point in $\coamoeba'_I \setminus \coamoeba_I$
we have a coamoeba point).

\subsection{The situation at $\infty$}
It is well-known that the geometry of amoebas at infinity (i.e.,
the ``tentacles'') can be characterized in terms of logarithmic
limit sets and tropical geometry, and thus amoebas form one of the 
building blocks of tropical geometry (for general background
on tropical geometry we refer
to \cite{maclagan-sturmfels,mikhalkin-amoebastrop,rst-2005}).
\iffullversion

For (large) $R > 0$ let $\mathcal{A}^{(R)}$ denote the scaled version
$\mathcal{A}^{(R)}_I \ := \ \frac{1}{R} \mathcal{A}_I \cap \Sph^{n-1}$,
where $\Sph^{n-1}$ denotes the $(n-1)$-dimensional unit sphere.
Extending this definition, 
the \emph{logarithmic limit set} $\mathcal{A}^{(\infty)}_I$ 
is the set of points $v \in \Sph^{n-1}$ such that there exists a sequence
$v_R \in \mathcal{A}^{(R)}_I$ with $\lim_{R \to \infty} v_R = v$.
For a polynomial $f = \sum_{\alpha} c_{\alpha} Z^{\alpha} \in \C[Z]$ 
denote by $\trop f := \bigoplus 0 \odot Z_1^{\alpha_1} \odot \cdots \odot Z_n^{\alpha_n}$
its tropicalization (with respect to the trivial valuation)
over the tropical semiring $(R,\oplus,\odot) := (\R \cup \{-\infty\},\max,+)$. 
Then (see \cite{maclagan-sturmfels,sturmfels-solving-2002}):
\fi

\begin{prop}
A vector $w \in \R^n \setminus \{0\}$ is contained in the tropical variety of $I$
if and only if the corresponding unit vector $\frac{1}{||w||}w$ is
contained in $\mathcal{A}^{(\infty)}_I$. Thus the tropical variety of $I$
coincides with the cone over the logarithmic limit set $\mathcal{A}^{(\infty)}_I$.
\end{prop}

\section{Approximations based on the real Nullstellensatz\label{se:nullstellensaetze}}

We study certificates of points in the complement of the amoeba
based on the real Nullstellensatz and compare them to 
existing statements in the literature (such as the toric Nullstellensatz
of Niculescu and Putinar). By imposing degree truncations this will
then yield a hierarchy of certificates of bounded degree.

We use the following real Nullstellensatz (see, e.g.,
\cite{bcr-98,prestel-delzell-2001}):

\begin{prop} \label{pr:realnullstellensatz}
For polynomials $g_1, \ldots, g_r \in \R[X]$
and $I:=\langle g_1, \ldots, g_r \rangle \subset \R[X]$
the following statements are equivalent:
\begin{itemize}
\item The real variety $\V_{\R}(I)$ is empty.
\item There exist a polynomial $G \in I$ and a sum of squares polynomial
  $H$ with
\[
  G + H + 1 \ = \ 0 \, .
\]
\end{itemize}
\end{prop}

Given $\lambda \in (0,\infty)^n$, the question if 
$\lambda$ is contained in the unlog amoeba ${\mathcal U}_I$ 
can be phrased as the real solvability of a real system of polynomial
equations.
For a polynomial $f \in \C[Z] = \C[Z_1, \ldots, Z_n]$ 
let $f^{\re}, f^{\im} \in \R[X,Y] = \R[X_1, \ldots, X_n, Y_1, \ldots, Y_n]$
be its real and imaginary parts, i.e.,
\[ 
  f(Z) \ = \ 
  f(X+iY) \ = \ f^{\re}(X,Y) + i \cdot f^{\im}(X,Y) \, .
\]

\noindent
We consider the ideal $I' \subset \R[X,Y]$ 
generated by the polynomials
\begin{equation}
  \label{eq:iprime}
   \{ f_j^{\re}, f_j^{\im} \, : \, 1 \le j \le r \} 
   \cup
   \big\{ X_k^2 + Y_k^2 - \lambda_k^2 \, : \, 1 \le k \le n \big\} \, . 
\end{equation}

\begin{cor}
\label{co:amoebacertificate}
Let $I = \langle f_1, \ldots, f_r \rangle$, and $\lambda \in (0,\infty)^n$.
Either the point $\lambda$ is contained in 
$\mathcal{U}_I$, or there exist a polynomial $G \in I' \subset \R[X,Y]$ 
and a sum of squares polynomial $H \in \R[X,Y]$ with
\begin{equation}
  \label{eq:amoebacertificate}
  G + H + 1 \ = \ 0 \, .
\end{equation}
\end{cor}

\begin{proof} For any polynomial $f \in \C[Z]$ it suffices to 
observe that a point $z=x+iy$ is contained in $\V(f)$ if and only
if $(x,y) \in \V_{\R}(f^{\re}) \cap \V_{\R}(f^{\im})$, and that
$|z_k| = \lambda_k$ if and only $(x,y) \in \V_{\R}(X^2+Y^2 - \lambda_k^2)$.
Then the statement follows from Proposition~\ref{pr:realnullstellensatz}.
\end{proof}

Corollary~\ref{co:amoebacertificate} states that 
for any point $\lambda \not\in \mathcal{U}_I$ there exists a certificate 
\begin{eqnarray}
	\sum_{j=1}^r p_j f_j^{\re} + 
	\sum_{j=1}^r p'_j f_j^{\im} +
	\sum_{k=1}^n q_k (X_k^2 + Y_k^2 - \lambda_k^2) + H + 1 & = & 0
\label{Equ:StandardApproach}  
\end{eqnarray}
with polynomials $p_j, p'_j, q_k$ and a sum of squares $H$. We refer to these certificates as \textit{certificates in the standard approach}.

We say that a certificate of the form \eqref{Equ:StandardApproach} is of \textit{degree at most} $d$ if the (total) degree of each summand in \eqref{Equ:StandardApproach} is at most $d$.

\begin{rem} By the following fact (which is easy to check), 
the sum of squares condition~\ref{eq:amoebacertificate} 
can also be stated shortly as 
\[
  -1 \text{ is a sum of squares in the quotient ring } \R[X,Y]/I' \, .
\]
\end{rem}

\medskip 

\begin{fact} \label{fa:sosquotientring}
\emph{(Parrilo \cite{parrilo-algebraicstructure}.)}
Let $I = \langle g_1, \ldots, g_r \rangle \subset \R[X]$
and $f \in \R[X]$.
There exist $p_1, \ldots, p_r \in \R[X]$ such that
\[
  f + \sum_i p_i g_i \text{ is a sum of squares in } \R[X]
\]
if and only if $f \text{ is a sum of squares in } \R[X]/I$.
\end{fact}
Any two of these equivalent conditions in Fact~\ref{fa:sosquotientring}
is a certificate for the nonnegativity of $f$ on the variety $I$.

Before stating a coamoeba version, we note the following normalization
properties.
Whenever it is needed for amoebas, 
we can assume that the point $\lambda$ in the amoeba membership problem is the all-1-vector
$\mathbf{1} \in \R^n$. Similarly, for the coamoeba membership problem we can assume that the investigated point is the origin $0 \in \T^n$. 

\medskip 

\begin{lemma} \label{le:normalization}
Let $I =  \langle f_1, \ldots, f_r \rangle$.
\begin{enumerate}
\item A point $(\lambda_1, \ldots, \lambda_n) \in (0,\infty)^n$ is contained in 
$\mathcal{U}_I$ if and only if $\mathbf{1} \in
\mathcal{U}_{\langle g_1, \ldots, g_r \rangle}$, where
\[
  g_j(Z_1, \ldots, Z_n) \ := \ f_j(\lambda_1 Z_1, \ldots, \lambda_n Z_n) \, ,
  \quad 1 \le j \le r \, .
\]
\item A point  $(z_1, \ldots, z_n)$ is contained in $\V(I)$ with
$\arg z_j = \mu_j$ if and only if the (nonnegative) real vector 
$y$ with $y_j \ := \ z_j e^{-i \mu_j}$
is contained in $\V(g_1, \ldots, g_r)$ where
\[
  g_j(Z_1, \ldots, Z_n) := f_j(Z_1 e^{i \mu_1}, \ldots, Z_n e^{i \mu_n}) \, ,
  \quad 1 \le j \le r \, .
\]
\end{enumerate}
\label{LemmaNormalform}
\end{lemma}

\begin{proof}
A point $(z_1, \ldots, z_n)$ is contained in $\V(I)$ 
with $|z_j| = \lambda_j$ if and only if the vector $y$ defined
by $y_j := z_j / \lambda_j$ is contained in $\V(g_1, \ldots, g_r)$ 
with $|y_j| = 1$. The second statement follows similarly.
\end{proof}

\begin{thm} Let $I = \langle f_1, \ldots, f_r \rangle$.
The point $(0, \ldots, 0)$ is contained in the complement of the
coamoeba $\coamoeba'_I$ if and only if there exists a 
polynomial identity
\begin{equation}
\label{th:coamoebasos}
  \sum_{j=1}^r c_j \cdot f_j(X^2,Y)^{\re}
  + \sum_{j=1}^r c_j' \cdot f_j(X^2,Y)^{\im}
  + \sum_{k=1}^n d_k \cdot Y_k + H + 1 \ = \ 0
\end{equation}
with polynomials $c_j, c_j', d_k \in \R[X,Y]$ and a sum 
of squares $H$. Here, $f_j(X^2,Y)$ abbreviates 
$f_j(X_1^2, \ldots, X_n^2,Y_1,\ldots,Y_n)$.
\label{Thm:CoamoebaCertificate}
\end{thm}

\begin{proof}
Note that the statement $0 \in \coamoeba'_I$ is
equivalent to $\{z=x+iy \in \C^n \, : \, z \in \V(I) \text{ and } 
  x_k \ge 0, \, y_k = 0, \, 1 \le k \le n \} \neq \emptyset$.
Moreover, observe that the condition $x_k \ge 0$ can be replaced
by considering $X_k^2$ in the arguments of $f_1, \ldots, f_r$.
Hence, by Proposition~\ref{pr:realnullstellensatz} the statement
$0 \not\in \coamoeba'_I$ is equivalent to the existence of a 
polynomial identity of the form~\eqref{th:coamoebasos}.
\end{proof}

Observe that in the proof the use of $\coamoeba'_I$ (rather than 
$\coamoeba_I$) allowed to use the basic Nullstellensatz (rather
than a Positivstellensatz, which would have introduced several
sum of squares polynomials).

The following variant of the Nullstellensatz approach will allow
to obtain degree bounds (see Section~\ref{se:specialcertificates}).
For vectors $\alpha(1), \ldots, \alpha(d) \in \N_0^n$ and coefficients
$b_1, \ldots, b_d \in \C^*$ let 
$f = \sum_{j = 1}^d b_j \cdot Z^{\alpha(j)} \in \C[Z]$. 
For any given values of $\lambda_1, \ldots, \lambda_n$ set
\[
  \mu_{j} \ := \ \lambda^{\alpha(j)} 
  \ =  \ \lambda_1^{\alpha(j)_1} \cdots \lambda_n^{\alpha(j)_n}
 \, , \quad 1 \le j \le d \, .
\]

If the rank of the matrix with columns $\alpha(1), \ldots, \alpha(d)$
is $n$ (i.e., the vectors $\alpha(1), \ldots, \alpha(d)$ span $\R^n$)
then the $\lambda$-values can be reconstructed uniquely from
the $\mu$-values. We come up with the following variant of
a Nullstellensatz.

Here, let $I := \langle f_1,\ldots,f_r \rangle$ such that $f_i$ is of the form $\sum_{j = 1}^{d_i} b_{i j} Z^{\alpha(i, j)}$ with $\alpha(i,j) \in \N^n_0$. Let $m_{i j}$ be the monomial 
$m_{i j} = Z^{\alpha(i,j)} = Z_1^{\alpha(i,j)_1} \cdots Z_n^{\alpha(i,j)_n}$. We consider the ideal $I^* \subset \R[X,Y]$ 
generated by the polynomials
\begin{equation}
   \{ f_i^{\re}, f_i^{\im} \, : \, 1 \le i \le r \}
   \cup 
   \left\{ (m_{i j}^{\re})^2 + (m_{i j}^{\im})^2 - \mu_{ij}^2 \, : \,1 \leq j \leq d_i \, , 1 \leq i \leq r  \right\} \, , 
   \label{eq:istar}
\end{equation}
where $\mu_{ij} = \lambda^{\alpha(i,j)}$.

\begin{cor}
\label{co:monomialbasedcertificate}
Let $I := \langle f_1,\ldots,f_r \rangle$, and assume that the set $\bigcup_{i = 1}^r \bigcup_{j = 1}^{d_i} \{\alpha(i,j)\}$ spans $\R^n$. 
Either a point $\lambda \in (0,\infty)^n$ is contained in 
$\mathcal{U}_I$, or there exist polynomials $G \in I^* \subset \R[X,Y]$ 
and a sum of squares polynomial $H \in \R[X,Y]$ with
\begin{equation}
  \label{eq:monomialbasedcertificate}
  G + H + 1 \ = \ 0 \, .
\end{equation}
\end{cor}

We refer to these certificates as \textit{certificates in the monomial approach}.\\

For hypersurface amoebas of real polynomials, 
the membership problem relates to the following
statement of Niculescu and Putinar \cite{niculescu-putinar-2010}.
Let 
$p = p(X,Y) \in \R[X_1, \ldots, X_n, $ $Y_1, \ldots, Y_n]$ be a real polynomial.
Then $p$ can be written as a complex polynomial $p(X,Y) = P(Z, \overline{Z})$
with $P \in \C[Z_1 \ldots, Z_n, \bar{Z}_1, \ldots, \bar{Z}_n]$
and $\overline{P(Z,\overline{Z})} = P(Z,\bar{Z})$. Note that there
exists a polynomial $Q \in \C[Z_1, \ldots, Z_n]$ with
\[
  p(x,y)^2 \ = \ |P(z,\bar{z})|^2 \ = \ |Q(z)|^2 \, \ \text{ for }
  z \in T^n \, ,
\]
where $T := \{ z \in \C \, : \, |z| = 1 \}$.

The following statement can be obtained
by applying the Nullstellensatz on the set
$\{ z = (x,y) \, : \, |q(z)|^2 = 1, \, |z_1|^2 = 1, \ldots, |z_n|^2 = 1 \}$,
then applying Putinar's Theorem \cite{putinar-93} on the multiplier
polynomial of $|q(Z)|^2$ (see \cite{niculescu-putinar-2010}).

\begin{prop}
\label{pr:niculescuputinar}
Let $q \in \C[Z_1, \ldots, Z_n]$. Then $q(z) \neq 0$
for all $z \in T^n$ if and only if there are complex
polynomials $p_1, \ldots, p_k, r_1, \ldots, r_l \in \C[Z_1, \ldots, Z_n]$
with
\begin{equation}
\label{eq:niculescuputinar}
  1 + |p_1(z)|^2 + \cdots + |p_k(z)|^2 \ = \ 
  |q(z)|^2 (|r_1(z)|^2 + \cdots + |r_l(z)|^2) \, , \quad z \in T^n \, .
\end{equation}
\end{prop}

Note that the statement is not an identity of polynomials, but an 
identity for all $z$ in the $n$-\emph{torus} $T^n$.

While Proposition~\ref{pr:niculescuputinar}
provides a nice structural result, due to the following reasons we
prefer Corollary~\ref{co:amoebacertificate} for actual computations.
In representation~\eqref{eq:niculescuputinar}, \emph{two} sums of squares 
polynomials (rather than just one as in~\eqref{eq:amoebacertificate})
are needed in the representation, and the degree is increased (by the 
squaring process). Moreover, the theorem is not really a representation 
theorem (in terms of an identity of polynomials), but an identity
over $T^n$; therefore in order to express this computationally, the
polynomials hidden in this equivalence (i.e., the polynomials
$1-|Z_1|^2, \ldots, 1 - |Z_n|^2$) have to be additionally used.

\subsection*{SOS-based approximations\label{se:sos}}

By putting degree truncations on the certificates, we can transform
the theoretic statements into effective algorithmic procedures for
constructing certificates. The idea of degree truncations in
polynomial identities follows the same principles of the degree truncations
with various types of Nullstellen- and Positivstellens\"atze in
\cite{lasserre-2001,dlmm-2008,parrilo-2003}.
\iffullversion
It is instructive to have a look at two simple examples first.

\begin{ex}
Let $f$ be the polynomial $f = Z + z_0$ with a complex constant 
$z_0 = x_0 + i y_0$. The ideal $I'$ of interest is defined by
\begin{eqnarray*}
  h_1 & := & f^{\re} \ = \ X + x_0 \, , \\
  h_2 & := & f^{\im} \ = \ Y + y_0 \, , \\
  h_3 & := & X^2+Y^2-\lambda^2 \, .
\end{eqnarray*}
For values of $\lambda \ge 0$ which correspond to points outside the 
amoeba (i.e., $\lambda^2 \neq x_0^2+y_0^2$),
we have $\V_{\C}(I) = \emptyset$
and thus the Gr\"obner basis $G$ of $\langle h_1,h_2,h_3 \rangle$
is $G = \{1\}$. The corresponding multiplier polynomials $p_j$
to represent~1 as a linear combination $\sum_j p_j h_j$
are
\[
  p_1 \ = \ \frac{-X+x_0}{x_0^2+y_0^2-\lambda^2} \, , \quad
  p_2 \ = \ \frac{-Y+y_0}{x_0^2+y_0^2-\lambda^2} \, , \quad
  p_3 \ = \ \frac{1}{x_0^2+y_0^2-\lambda^2} \, .
\]
Hence,
in particular, $-1$ can be written as a sum of squares in the quotient
ring $\R[X] / I'$.
The necessary degree with regard to equation~\eqref{eq:amoebacertificate}
is just~2.

For $\lambda^2 = x_0^2+y_0^2$, the Gr\"obner basis (w.r.t.\ a lexicographic 
variable ordering with $X \succ Y$) is
\[
  X + x_0 , \; Y + y_0 \, .
\]
The point $(-x_0,-y_0)$ is contained in $\V_{\R}(I')$; thus in this case
there does not exist a Nullstellen-type certificate.
\end{ex}

\begin{ex}
Consider the polynomial $f = Z_1 + Z_2 + 5$ with $Z_j = X_j + i Y_j$.
The ideal $I'$ of interest is defined by
\begin{eqnarray*}
  h_1 := X_1 + X_2 + 5 \, , \ h_2 := Y_1 + Y_2 \, , \ h_3 := X_1^2+Y_1^2-\lambda_1^2 \, , \ h_4 := X_2^2+Y_2^2-\lambda_2^2 \, .
\end{eqnarray*}
Consider $\lambda_1 = 2$, $\lambda_2 = 3$.
Using a lexicographic ordering with
$X_1 \succ X_2 \succ Y_1 \succ Y_2$, a Gr\"obner basis is
\[
  Y_2^2 \, , \; Y_1 + Y_2 \, , \; X_2 + 3 \, , \; X_1 + 2 \, .
\]
The standard monomials are 1 and $Y_2$.
It is easy to see that $-1$ is not a sum of squares in the quotient ring, which reflects the fact that $(2,3) \in \cU_{I}$.

Consider now the choice $\lambda_1 = 1$ and $\lambda_2 = 2$. Using lexicographic ordering again, the Gr\"obner basis is
$$25 Y_2^2 + 96 \, , \; Y_1 + Y_2 \, , \; 5 X_2 + 14 \, , \; 5 X_1 + 11 \, .$$
Hence, $25 Y_2^2 \equiv -96 \mod I'$, which gives the sum of squares identity
\begin{eqnarray*}
	\left(\frac{5}{\sqrt{96}} Y_2\right)^2 & \equiv & -1 \mod I' \, ,
\end{eqnarray*}
and thus shows $(1,2) \notin \cU_I$.
\end{ex}

\fi

Using the degree truncation approach for sums of squares we can for a given ideal $I = \langle f_1,\ldots,f_r \rangle$ define
\begin{eqnarray*}
	C_t \ := \ \{\lambda \in (0,\infty)^n \setminus \cU_I & : & \text{there exists a certificate of the form } \eqref{Equ:StandardApproach}  \\
	& & \text{for } \lambda \text{ of degree } \leq 2t\}.
\end{eqnarray*}

Similarly, for coamoebas $\cC_I$ let $D_t$ be the subsets of its complement obtained by the degree truncation.
These sequences are the basis of the effective implementation (see Section~\ref{se:computing}).

\begin{thm} 
\label{th:boundeddegree}
Let $I = \langle f_1, \ldots, f_r \rangle$ and $t_0 := \max_j \lceil \deg f_j/2 \rceil$.
The sequence $(C_t)_{t \ge t_0}$ converges pointwise to the complement 
of the unlog
amoeba $\mathcal{U}_I$, and it is monotone increasing in the set-theoretic
sense, i.e. $C_t \subset C_{t+1}$ for $t \ge t_0$.

Similarly, the sequence $(D_t)_{t \ge t_0}$ converges pointwise to the complement 
of the coamoeba $\coamoeba_I$, and it is monotone increasing in the set-theoretic
sense, i.e. $D_t \subset D_{t+1}$ for $t \ge t_0$.
\end{thm}

\begin{proof} For any given point $z$ in the complement of the amoeba there 
exists a certificate of minimal degree, say $d$. For $t < \lceil d/2 \rceil$
the point $z$ is not contained in $C_t$ and for $t \ge \lceil d/2 \rceil$ the
point $z$ is contained in $C_t$. In particular, the relaxation process is
monotone increasing. And analogously for coamoebas.
\end{proof}

\begin{rem} A similar result holds for the monomial approach. Namely, for a given ideal 
$I = \langle f_1,\ldots,f_r \rangle$, $t_0 := \max_j \lceil \deg f_j/2 \rceil$, and
$I^*$ generated by the polynomials~\eqref{eq:istar}, the sets
\begin{eqnarray*}
	C_t^* \ := \ \{\lambda \in (0,\infty)^n \setminus \cU_I & : & \text{there exists a certificate of the form } G + H + 1 = 0 \\
	& & \text{with } G \in I^* \text{ and } H \text{ SOS for } \lambda \text{ of degree } \leq 2t\}
\end{eqnarray*}
$(t \ge t_0)$
converge pointwise to the complement of $\mathcal{U}_I$, and, set-theoretically,
this sequence is monotone increasing.
\end{rem}

It is well-known (and at the heart of current developments in optimization of
polynomial functions, see \cite{lasserre-2001,parrilo-2003} or the
survey \cite{laurent-2009-survey})
that SOS conditions with degree constraints of the form \eqref{Equ:StandardApproach} can be phrased 
as semidefinite programs.
Finding an (optimal) positive semidefinite matrix within an affine 
linear variety is known as semidefinite programming 
(see e.g.\ \cite{sdp-handbook} for a comprehensive treatment).
Semidefinite programs can be solved efficiently both in theory
and in practice.

Precisely, any sum-of-squares polynomial $H$ can be expressed as
$M Q M^T$, where $Q$ is a symmetric positive semidefinite matrix
(abbreviated $Q \succeq 0$) and $M$ is a vector of monomials.

Similarly, by the degree restriction the linear combination 
in~\eqref{eq:amoebacertificate}
or~\eqref{eq:monomialbasedcertificate}
can be integrated into the semidefinite formulation by a
comparison of coefficients.

\section{Special certificates\label{se:specialcertificates}}

For a certain class of amoebas, we can provide some explicit classes of 
Nullstellensatz-type certificates. As a first warmup-example, we
illustrate some ideas for constructing special certificates systematically
for linear amoebas in the standard approach.
Then we show how to construct special certificates for the
monomial-based approach.

In this section we concentrate on the case of hypersurface amoebas.

\subsection*{Linear amoebas in the standard approach}

Let $f = a Z_1 + b Z_2 + c$ be a general linear polynomial
in two variables with real coefficients $a,b,c \in \R$.
We consider certificates of the form~\eqref{eq:amoebacertificate}
based on the third binomial formula $(\alpha + \beta)(\alpha - \beta) = \alpha^2 - \beta^2$,
and we use the sums of squares $(X_1 - X_2)^2$ and $(Y_1 - Y_2)^2$.
For simplicity assume $a,b > 0$. Setting
\begin{eqnarray*}
  G_1 & := & (a X_1+b X_2-c)(a X_1+b X_2+c)
  +(a Y_1+b Y_2)(a Y_1+b Y_2) \\
     & & -(a^2+a b)(X_1^2+Y_1^2-\lambda_1^2)
  -(b^2+a b)(X_2^2+Y_2^2-\lambda_2^2) \, , \\
  H_1 & := &  a b (X_1-X_2)^2+ab(Y_1-Y_2)^2 
\end{eqnarray*}
the sum $G_1 + H_1$ simplifies via elementary cancellation to
\begin{eqnarray}
  \label{eq:linearcertificate1}
  \gamma_1 & := & (a^2+ab) \lambda_1^2 + (b^2+ab) \lambda_2^2 - c^2 \, .
\end{eqnarray}
Assume that the point $(\lambda_1,\lambda_2)$ is not contained in
the unlog amoeba $\mathcal{U}_f$.
In order to obtain the desired polynomial identity~\eqref{eq:amoebacertificate} 
certifying containment in the complement of $\mathcal{U}_f$, we require $\gamma_1$ to be negative. In that case we have
\begin{eqnarray*}
	\frac{1}{|\gamma_1|} \left(G_1 + H_1\right) + 1 & = & 0 \, ,
\end{eqnarray*}
which gives the polynomial identity~\eqref{eq:amoebacertificate}. Analogously, we obtain for
\begin{eqnarray*}
  G_2 & := & (-a X_1+b X_2+c)(a X_1+b X_2+c)
  +(-a Y_1+b Y_2)(a Y_1+b Y_2) \\
     & & a^2(X_1^2+Y_1^2-\lambda_1^2)
  -(b^2+bc)(X_2^2+Y_2^2-\lambda_2^2) \, , \\
  H_2 & := &  bc(X_2 - 1)^2 + bcY_2^2
\end{eqnarray*}
via elementary cancellation
\begin{eqnarray*}
	\gamma_2 & := & G_2 + H_2 \ = \ (b^2 + bc) \lambda_2^2 + c^2 + bc - a^2 \lambda_1^2 \, ,
\end{eqnarray*}
and, symmetrically,
\begin{eqnarray*}
	\gamma_3 & := & (a^2 + ac) \lambda_1^2 + c^2 + ac - b^2 \lambda_2^2 \, .
\end{eqnarray*}

\begin{ex}
Let $a=1$, $b=2$, $c=5$.
The curve (in $\lambda_1,\lambda_2$)
given by \eqref{eq:linearcertificate1} has
a logarithmic image that is shown in Figure~\ref{fi:firstpicture}.
Analogous special certificates can be
obtained within the two other complement components.
\end{ex}

\ifpictures
\begin{figure}[ht]
\[
    \includegraphics[scale=1]{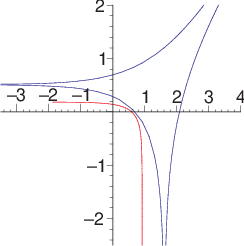}
\]

\caption{The boundary of an amoeba of a linear polynomial (blue) and the logarithmic image of the boundary 
of $R_1 = \{\lambda \in \R^2 \ : \ (a^2 + ab) \lambda_1^2 + (b^2+ab) \lambda_2^2 - c^2 < 0\}$, for which the special certificates of degree~2 exist (red).}
\label{fi:firstpicture}
\end{figure}
\fi

Since (by homogenizing the polynomial) there
is a symmetry, we obtain similarly an approximation of the other
two complement components. Hence we have:

\begin{lemma}
Let $f = aZ_1 + bZ_2 +c$ with $a,b,c \in \R$, and set
\begin{eqnarray*}
	R_1 & = & \{\lambda \in \R^2 \ : \ (a^2 + ab) \lambda_1^2 + (b^2+ab) \lambda_2^2 - c^2 < 0\} \, , \\
	R_2 & = & \{\lambda \in \R^2 \ : \ (b^2 + bc) \lambda_2^2 + c^2 + bc - a^2 \lambda_1^2 < 0\} \, , \\
	R_3 & = & \{\lambda \in \R^2 \ : \ (a^2 + ac) \lambda_1^2 + c^2 + ac - b^2 \lambda_2^2 < 0\} \, .
\end{eqnarray*}
Then $R_j \cap \cU_f = \emptyset$ for $1 \leq j \leq 3$, and for any $\lambda \in R_1 \cup R_2 \cup R_3$ there exists a certificate \eqref{Equ:StandardApproach} of degree at most two.
\end{lemma}

\subsection*{The monomial-based approach}

For the monomial-based approach based on Corollary~\ref{co:monomialbasedcertificate}
we can provide special certificates for a much more general class.
Our point of departure is Purbhoo's \textit{lopsidedness} criterion
\cite{purbhoo-2008} which guarantees that a point belongs to the 
complement of an amoeba $\mathcal{A}_f$.
In particular, we can
provide degree bounds for these certificates.

In the following let $\alpha(1), \ldots, \alpha(d) \in \N_0^n$ span $\R^n$
and $f = \sum_{j=1}^d b_j Z^{\alpha(j)} \in \C[Z]$ with monomials
$m_j := Z^{\alpha(j)} = Z_1^{\alpha(j)_1} \cdots Z_n^{\alpha(j)_n}$. 
For any point $v \in \R^n$ define $f\{v\}$ as 
the following sequence of numbers in $\R_{> 0}$,
\[
  f\{v\} \ := \ \left(|b_1 m_1(\Log^{-1}(v))|,\ldots,|b_d m_d(\Log^{-1}(v))|\right) \, .
\]
A list of positive real numbers is called \textit{lopsided} if one of the numbers is greater than the sum of all the others. We call a point 
$v \in \R^n$ \emph{lopsided}, if the sequence $f\{v\}$ is lopsided. 
Furthermore set
\[
  \mathcal{L}\mathcal{A}(f) \ := \ \{v \in \R^n \ : \ f\{v\} \text{ is not lopsided} \}.
\]
It is easy to see that $\mathcal{A}_f \subset \mathcal{L}\mathcal{A}(f)$ with $\mathcal{A}_f \neq \mathcal{L}\mathcal{A}(f)$ in general. 
In the following way the amoeba $\mathcal{A}_{f}$ can be approximated based on
the lopsidedness concept. For $r \ge 1$ let
\begin{eqnarray*}
	\tilde{f}_r(Z)	& :=	& \prod_{k_1 = 0}^{r-1} \cdots \prod_{k_n = 0}^{r-1} f\left(e^{2\pi i k_1 / r} Z_1,\ldots,e^{2\pi i k_n / r} Z_n\right)  \label{EquIterPolynLops}\\
				& =	& \Res \left(\Res \left(\ldots \Res(f(U_1 Z_1,\ldots,U_n Z_n),U_1^r - 1),\ldots, U_{n-1}^r - 1\right),U_n^r - 1\right) \, , \nonumber
\end{eqnarray*}
where $\Res$ denotes the resultant of two (univariate) polynomials.
Then the following theorem holds.
\begin{thm}\emph{(Purbhoo \cite[Theorem 1]{purbhoo-2008})}
\label{th:purbhoo}
\begin{enumerate}
	\item[(a)] For $r \to \infty$ the family $\mathcal{L}\mathcal{A}(\tilde{f}_r)$ converges uniformly to $\mathcal{A}_f$. There exists an integer $N$ such that to compute $\mathcal{A}_f$ within $\varepsilon > 0$, it suffices to compute $\mathcal{L}\mathcal{A}(\tilde{f}_r)$ for any $r \geq N$. Moreover, $N$ depends only on $\varepsilon$ and the Newton polytope (or degree) of $f$ and can be computed explicitly from these data.
	\item[(b)] For an ideal $I \subset \C[Z]$ a point $v \in \R^n$ is in the amoeba $\mathcal{A}_I$ if and only if $g\{v\}$ is not lopsided for every $g \in I$.
\end{enumerate}
\label{ThmPurb1}
\end{thm}

Based on Theorem~\ref{th:purbhoo} one can devise a converging sequence of 
approximations for the amoeba.
Note, however, that the lopsidedness criterion is not a Nullstellensatz in a strict sense since 
it does not provide a polynomial identity certifying membership in the complement
of the amoeba.

The aim of this section is to figure out how our SOS approximation is related to 
the lopsidedness and transform the lopsidedness certificate
into a certificate for the Nullstellens\"atze presented in 
Section~\ref{se:nullstellensaetze}.

By Lemma~\ref{LemmaNormalform} we can assume that the point
$\lambda$, whose membership to an unlog amoeba $\cU_f$ shall be decided, is the all-1-vector {\bf 1}. 
In this situation lopsidedness means that there is an index
$j \in \{1, \ldots, d\}$ with $|b_j| > \sum_{i \neq j} |b_i|$.
If the lopsidedness condition is satisfied in $v$, then
the following statement provides a certificate of the form $G+H+1$ and
bounded degree.

Corresponding to the definition of $I^*$ in~\eqref{eq:istar},
let the polynomials $s_1, \ldots, s_{d+2}$ be defined by
\[
  s_i \ = \ \left(\frac{b_i^{\re}}{|b_i|} \cdot \left(Z^{\alpha(i)}\right)^{\re}\right)^2 + \left(\frac{b_i^{\im}}{|b_i|} \cdot \left(Z^{\alpha(i)}\right)^{\im}\right)^2 - 1 \, , \quad 1 \le i \le d \, ,
\]
and $s_{d+1} = f^{\re}$, $s_{d+2} = f^{\im}$.

\begin{thm} \label{th:certificateforlopsided}
If the point $\lambda = {\bf 1}$ is contained in the complement of 
$\mathcal{U}_f$ with $f\{0\}$ being lopsided with dominating element
$|m_1(\mathbf{1})|$, then there exists a certificate 
of (total) degree $2 \cdot \deg(f)$ which is given by
\begin{equation}
  \label{eq:certificateforlopsided}
  \sum_{i=1}^{d+2} s_i g_i + H + 1 \ = \ 0 \, ,
\end{equation}
where
\begin{eqnarray*}
  g_1	& = & |b_1|^2 \, , \quad g_i \ = \ - |b_i| \cdot \sum_{k = 2}^d |b_k| \, , \quad 2 \leq i \leq d\, , \\
  g_{d+1} & = & \left(-b_1 \cdot Z^{\alpha(1)} + \sum_{i = 2}^d b_i \cdot Z^{\alpha(i)}\right)^{\re} \, , \quad
  g_{d+2} \ = \ \left(-b_1 \cdot Z^{\alpha(1)} + \sum_{i = 2}^d b_i \cdot Z^{\alpha(i)}\right)^{\im} \, , \\
        H & = & \sum_{2 \le i < j \le d} |b_i| \cdot |b_j| \cdot \left(\frac{b_i^{\re}}{|b_i|} \cdot \left(Z^{\alpha(i)}\right)^{\re} - \frac{b_j^{\re}}{|b_j|} \cdot \left(Z^{\alpha(j)}\right)^{\re} \right)^2 \\
                      & & + |b_i| \cdot |b_j| \cdot \left(\frac{b_i^{\im}}{|b_i|} \cdot \left(Z^{\alpha(i)}\right)^{\im} - \frac{b_j^{\im}}{|b_j|} \cdot \left(Z^{\alpha(j)}\right)^{\im} \right)^2 \, . 
\end{eqnarray*}
\label{ThmCertLopsidedness}
\end{thm}

\begin{proof} By the third binomial formula $(\alpha+\beta) \cdot (\alpha-\beta) = \alpha^2 - \beta^2$, substituting the polynomials $s_i$ and $g_j$ into 
$s_{d+1} g_{d+1} + s_{d+2} g_{d+2}$ yields
\[
  - \left(b_1^{\re} \cdot \left(Z^{\alpha(1)}\right)^{\re}\right)^2 + \left( \sum_{i = 2}^d (b_i \cdot Z^{\alpha(i)})^{\re} \right)^2
  - \left(b_1^{\im} \cdot \left(Z^{\alpha(1)}\right)^{\im}\right)^2 + \left( \sum_{i = 2}^d (b_i \cdot Z^{\alpha(i)})^{\im} \right)^2 \, . 
\]
Adding $g_1 s_1$ and the SOS term $H$ yields
$$
	-|b_1|^2 + \left(\sum_{j = 2}^d \left(\frac{b_j^{\re}}{|b_j|} \cdot \left(Z^{\alpha(j)}\right)^{\re}\right)^2 + \left(\frac{b_j^{\im}}{|b_j|} \cdot \left(Z^{\alpha(j)}\right)^{\im}\right)^2\right) \cdot \left(|b_j| \cdot \sum_{k = 2}^d |b_k|\right).
$$
Hence, the expression $\sum_{i=1}^{d+2} s_i g_i + H$ 
in~\eqref{eq:certificateforlopsided} in total results in 
\[
 - |b_1|^2 + \left(\sum_{i = 2}^d |b_i|\right)^2 \, .
\]
Since all $|b_i| \geq 0$ and by our assumption of
lopsidedness with dominating term $|m_1(\mathbf{1})|$,
this is the certificate we wanted to obtain.
By rescaling, we can bring the constant to $-1$.
\end{proof}

We say that there exists a certificate for a point $w$ in the complement
of the (log) amoeba $\mathcal{A}_f$ if there exists a certificate
for the point $\mathbf{1}$ in the complement of the amoeba $\mathcal{U}_g$
in the sense of Theorem~\ref{th:certificateforlopsided}, where
$g$ is defined as in Lemma~\ref{le:normalization} 
and $\lambda_i := |\log^{-1}(w_i)|$.

\begin{cor} \label{co:certificate} Let $r \in \N$. 
\begin{enumerate}
\item For any $w \in \R^n \setminus \mathcal{L}\mathcal{A}(\tilde f_r) \subset \R^n \setminus \mathcal{A}_f$ there exists a certificate
of degree at most $2 \cdot r^n \cdot \deg(f)$ which can be computed explicitly.
\item The certificate determines the order of the complement component 
to which $w$ belongs.
\end{enumerate}
\end{cor}

\begin{proof}
By definition of $g$, we have $w \in \mathcal{A}_f$ if and only
if $\mathbf{1} \in \mathcal{U}_g$. Further $\mathbf{1}$ belongs to 
$\mathcal{L}\mathcal{A}(\tilde{g}_r)$ if and only if 
$\tilde{g}_r\{0\}$ is not lopsided. Applying 
Theorem~\ref{ThmCertLopsidedness} on the function $\tilde{g}_r$
yields a certificate for $w$ in the log amoeba $\mathcal{A}_f$.
Since we have $\tdeg(\tilde{g}_r) = \tdeg(g) \cdot r^n =
\tdeg(f) \cdot r^n$ due to the definition of $\tilde g_r$ and
$g$ the result follows.

For the second statement, note that passing over from $f$ to
$g$ does not change the order $\nu$ of any point in the complement
of the amoebas. Now it suffices to show that the dominating term
(which occurs in a distinguished way in the certificate)
determines the order of the complement component. The latter
statement follows from Purbhoo's result that if 
$w \not\in \mathcal{L}\mathcal{A}(\hat{f}_r)$ and the order of 
the complement component $w$ belongs to is $\alpha(i)$ then 
the dominant term in $\tilde{f}_r$ has the exponent 
$r^n \cdot \alpha(i)$ (see \cite[Proposition 4.1]{purbhoo-2008}).
\end{proof}

\begin{cor} \label{th:hyperplaneamoebas}
For linear hyperplane amoebas in $\R^n$, any point in the complement
of the amoeba has a certificate whose sum of squares is a sum of
squares of affine functions.
\end{cor}

\begin{proof}
By the explicit characterization of linear hyperplane amoebas 
in~\cite[Corollary 4.3]{fpt-2000}, any point
in the complement is lopsided.
Hence, the statement follows from Theorem~\ref{th:certificateforlopsided}.
\end{proof}

\subsection*{Simplified expressions}

 From a slightly more general point of view, the monomial-based certificates
can be seen as a special case of the following construction.
Whenever the defining polynomials of a variety originate from
simpler polynomials with algebraically independent monomials,
then the approximation of the amoeba can be simplified. 

For an ideal $I$ let $V := \V(I) \subset (\C^*)^n$ be its
subvariety in $(\C^*)^n$.
Let $\gamma_1,\ldots,\gamma_k$ be $k$ monomials in $n$ variables,
say, $\gamma_i = Y^{\alpha{(i)}} = Y_1^{\alpha{(i)}_{1}} Y_2^{\alpha{(i)}_{2}}  \cdots  Y_n^{\alpha{(i)}_{n}} $,
where $\alpha{(i)} = (\alpha{(i)}_{1} ,\ldots, \alpha{(i)}_{n} ) \in \Z^n$.
They define a homomorphism $\gamma$ of algebraic groups
 from $(\C^*)^n$ to $(\C^*)^k$.
For any subvariety $W$ of $(\C^*)^k$, the inverse image 
$\gamma^{-1}(W)$ is a subvariety of $(\C^*)^n$. 
Note that the map $\gamma$ is onto if and only if the
vectors $\alpha{(1)}, \ldots,\alpha{(k)}$ are linearly independent
(see \cite[Lemma 4.1]{theobald-expmath-2002}).

Let $J$ be an ideal with $\V(J) = \gamma^{-1}(V)$.
If the map $\gamma$ is onto, then computing the amoeba 
of $J$ can be reduced to the computation of the amoeba of $I$.
Let $\gamma'$ denote the restriction of $\gamma$ to the multiplicative
subgroup $(0,\infty)^n$. Then the following diagram is
a commutative diagram of multiplicative abelian groups:
\[
\begin{matrix}
(\C^*)^n & \buildrel \gamma \over  \longrightarrow & (\C^*)^k \\
\downarrow & & \downarrow \\
(0,\infty)^n &  \buildrel \gamma' \over \longrightarrow & (0,\infty)^k \\
\end{matrix}
\]
where the  vertical maps are taking coordinate-wise absolute value.
For vectors $p = (p_1,\ldots,p_n)$ in $ (\C^*)^n$
we write $\,|p| = (|p_1|, \ldots, |p_n|) \in (0,\infty)^n$,
and similarly for vectors of length $k$. Further, for
$V \subset (\C^*)^n$ let $|V| := \{ |p| \, : \, p \in V\}$.
If the map $\gamma$ is onto then $|\gamma^{-1}(V)| = \gamma'^{-1}(|V|)$
(see \cite{theobald-expmath-2002}).

\begin{thm} \label{th:transfercertificates}
If a point outside of an unlog amoeba $\mathcal{U}_I$ has a certificate
of total degree $d$ then a point outside of the unlog $\mathcal{U}_J$
has a certificate of degree $d \cdot D$, where $D$ is the maximal
total degree of the monomials $\gamma_1, \ldots, \gamma_k$.
\end{thm}

In particular, this statement applies to the certificates from
Statements \ref{co:amoebacertificate} and \ref{co:monomialbasedcertificate}.
\iffullversion

\begin{proof}
Let $p$ be a point outside of the unlog amoeba of $V \subset (\C^*)^n$ 
which has a certificate of total degree $d$. By 
Corollary~\eqref{eq:amoebacertificate},
the certificate consists of a polynomial $G(X,Y)$ in the real ideal
$I' \subset \R[X,Y]$ from~\eqref{eq:iprime}
and by real sums of squares of polynomials in $\R[X,Y]$. 
For the polynomials in the ideals, we observe that
the realification process carries over to the substitution
process. W.l.o.g.\ we can assume that $\gamma_i$ is 
a product of just two factors. Then, with $Z = Z_1 + i Z_2$, $Z = PQ$
we have
$Z_1 + i Z_2 = (P_1 + i P_2)(Q_1 + i Q_2)$ and use the
real substitutions
$Z_1 \equiv P_1 Q_1 - P_2 Q_2$,
$Z_2 \equiv P_1 Q_2 + P_2 Q_1$.
And in the same way the real sum of squares remain real sums of 
squares (of the polynomials in $P_i$, $Q_j$) after substituting. 
\end{proof}

\begin{example}
Let $\mathbb{G}_{1,3}$ denote the Grassmannian
of lines in 3-space, which is the variety in $\P_{\C}^5$, 
defined by
\[
  P_{01} P_{23} \, - \,
 P_{02} P_{13} \, + \,
 P_{03} P_{12} \, = \, 0 \, ,
\]
which we consider as a subvariety of $(\C^*)^6$.
The three terms in this quadratic equation involve distinct variables and
hence correspond to linearly independent exponent vectors. Note that
$\mathbb{G}_{1,3}$ equals $\,\gamma^{-1}(V)\,$ where
\[
\gamma : (\C^*)^6 \, \rightarrow \, (\C^*)^3 \, , \,\,
(p_{01},p_{02},p_{03},p_{12},p_{13},p_{23}) \,\,\mapsto \,
( p_{01} p_{23}, \,
 p_{02} p_{13} ,\,
 p_{03} p_{12} )
\]
and $V$ denotes the plane in $3$-space defined by the linear equation
$X - Y + Z =  0$.
Since by Theorem~\ref{th:hyperplaneamoebas} any point in the complement
has a certificate of degree~2, 
Theorem~\ref{th:transfercertificates} implies that every point in the
complement of the Grassmannian amoeba has a certificate of degree~4.
\end{example}
\fi

\section{Computing the relaxations via semidefinite programming\label{se:computing}}

We close the paper by providing some computational results in order
to confirm the validity of our approach. The subsequent computations 
have been performed on top of {\sc SOSTools} \cite{sostools}
which is a {\sc Matlab} package for computing sums of squares based
computations. The SDP package underlying {\sc SOSTools} is {\sc SeDuMi} \cite{sedumi}.

\begin{figure}
\[
  \includegraphics[scale=0.55]{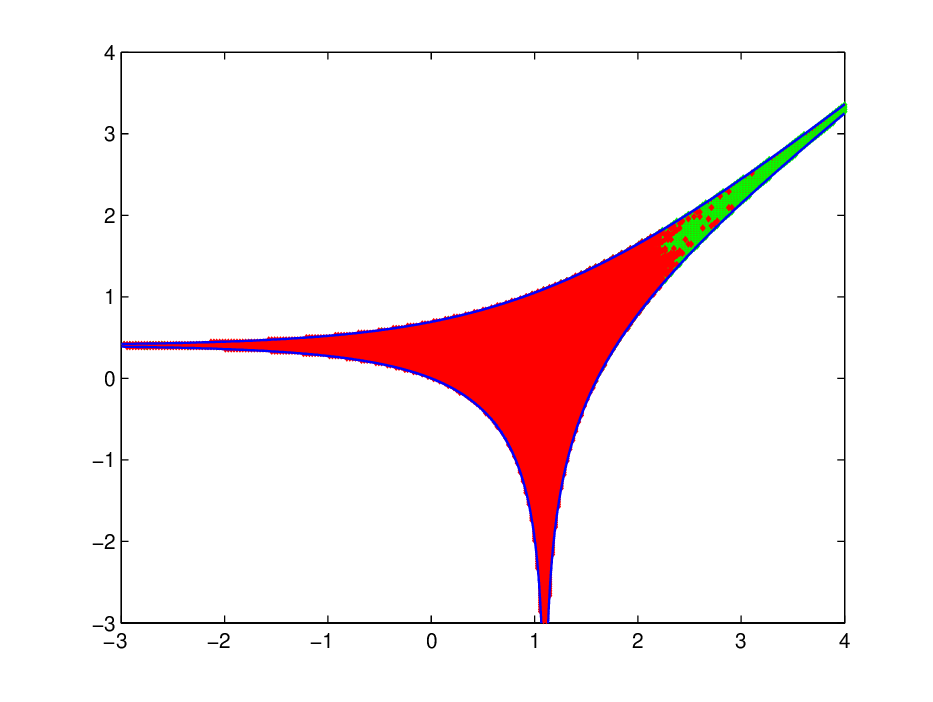}
\]
	\caption{SOS certificates of linear amoebas restricted to degree two. The
  red (dark) points represent infeasible SDPs, and in the green (light) points numerical instabilities 
  were reported in the computations.}
	\label{FigureLinearVerification}
\end{figure}

\begin{example}
For the test case of a linear polynomial $f := Z_1 + 2Z_2 + 3$,
the boundary contour of the amoeba $\mathcal{A}_f$ can be explicitly
described, and it is given by the curves
\begin{eqnarray*}
		\exp(z_1)	& =	& 2\cdot \exp(z_2) + 3 \, , \\
	2 \cdot \exp(z_2)	& =	& \exp(z_1) + 3 \, , \\
		3	& =	& \exp(z_1) + 2 \cdot \exp(z_2) \, ,
\end{eqnarray*}
see \cite{fpt-2000}.
We compute the amoeba of $f$ with our SDP via {\sc SOSTools} on a grid of size 
$250 \times 250$ lattice points in the area $[-3,4]^2$. 
In the SDP we restrict to polynomials of degree two. 
By Theorem~\ref{th:certificateforlopsided}, the approximation is exact in that case
(up to numerical issues). Figure \ref{FigureLinearVerification} visualizes
the SDP-based computation of the SOS certificates.
In the figure, at the white outer regions, certificates are found.
At the dark red points the SDP is infeasible; at the green points, no feasibility
is proven, but numerical instabilities are reported by the SDP solver.
The aspect of numerical stability of our SOS-based amoeba computations and
of general SOS computations is an important issue in convex
algebraic geometry which deserves further study. See \cite{Lofberg:Parrilo,Sun:Masterthesis} for existing work in this direction, based on choosing different bases.
\label{ex:LinearVerification}
\end{example}

\begin{ex}
As in Figure \ref{FigureExampleAmoeba} we consider the class of polynomials 
$f := Z_1^2 Z_2 + Z_1 Z_2^2 + c \cdot Z_1 Z_2 + 1$ with some constant $c \in \R$.
We use the monomial-based approach from Corollary~\ref{co:monomialbasedcertificate}.
In order to compute whether a given point $(\lambda_1,\lambda_2) \in (0,\infty)^2$ is contained
in the unlog amoeba $\mathcal{U}_f$, we have to consider the polynomials
\begin{eqnarray*}
	h_1 & = & (X_1^2X_2 - 2X_1Y_1Y_2 - Y_1^2 \cdot X_2) + (X_1X_2^2 - X_1Y_2^2 - 2Y_1X_2Y_2) \\
& & + c \cdot (X_1X_2 - Y_1Y_2) + 1 \, , \\
	h_2 & = & (X_1^2Y_2 + 2X_1Y_1X_2 - Y_1^2Y_2) + (2X_1X_2Y_2 + Y_1X_2^2 - Y_1Y_2^2) + c \cdot (X_1Y_2 + X_2Y_1) \, , \\
	h_3 & = & (X_1^2X_2 - 2X_1Y_1Y_2 - Y_1^2X_2)^2 + (X_1^2Y_2 + 2X_1Y_1X_2 - Y_1^2Y_2)^2 - (\lambda_1^2 \cdot \lambda_2)^2 \, , \\
	h_4 & = & (X_1X_2^2 - X_1Y_2^2 - 2Y_1X_2Y_2)^2 + (2X_1X_2Y_2 + Y_1X_2^2 - Y_1Y_2^2)^2 - (\lambda_1 \cdot \lambda_2^2)^2 \, , \\
	h_5 & = & (X_1X_2 - Y_1Y_2)^2 + (X_1Y_2 + X_2Y_1)^2 - (\lambda_1 \cdot \lambda_2)^2 \, .
\end{eqnarray*}

For the case $c=2$ and $c = -4$ we investigate $160 \times 160$ points in the 
area $[-4,4] \times [-5,3]$ and restrict the polynomials multiplied with the 
constraints to degree three, and thus restrict ourselves to a degree bound of six. The resulting amoeba $\mathcal{A}_f$ is depicted 
in Figure \ref{FigureExampleSOSAmoeba}. 
At the white points the SDP is feasible and thus these points belong to the complement component.
At the orange points the SDP is recognized as feasible with numerical issues (within a pre-defined 
range). At the black points in the center the SDP was infeasible without and at the turquoise points 
with numerical issues reported. At the red points in the upper right corner the program stopped due to exceeding numerical problems.
The union of the central black, the turquoise and part of the orange points provides the (degree bounded) approximation of the amoeba.
\label{Exa:GenusOne}
\end{ex}

\ifpictures
\begin{figure}[ht]
\includegraphics[width=0.45\linewidth]{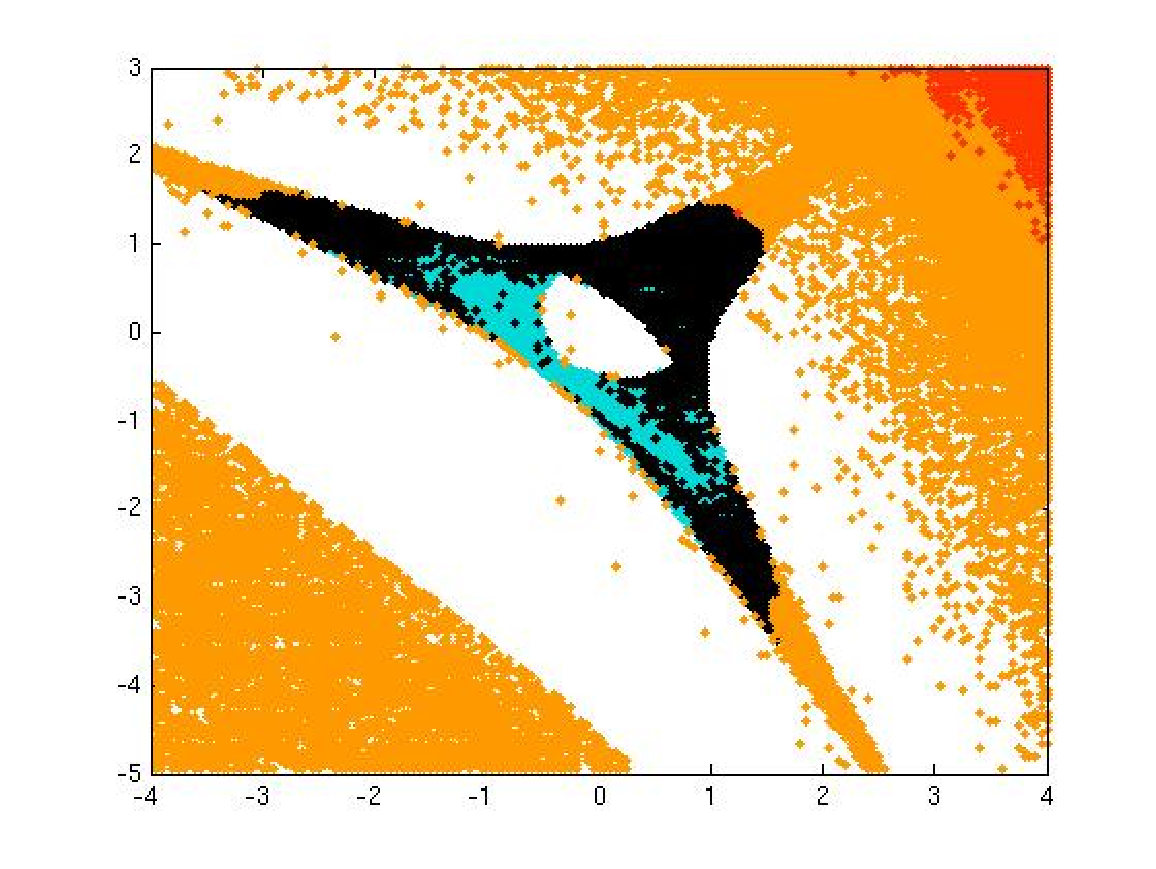} 
\includegraphics[width=0.45\linewidth]{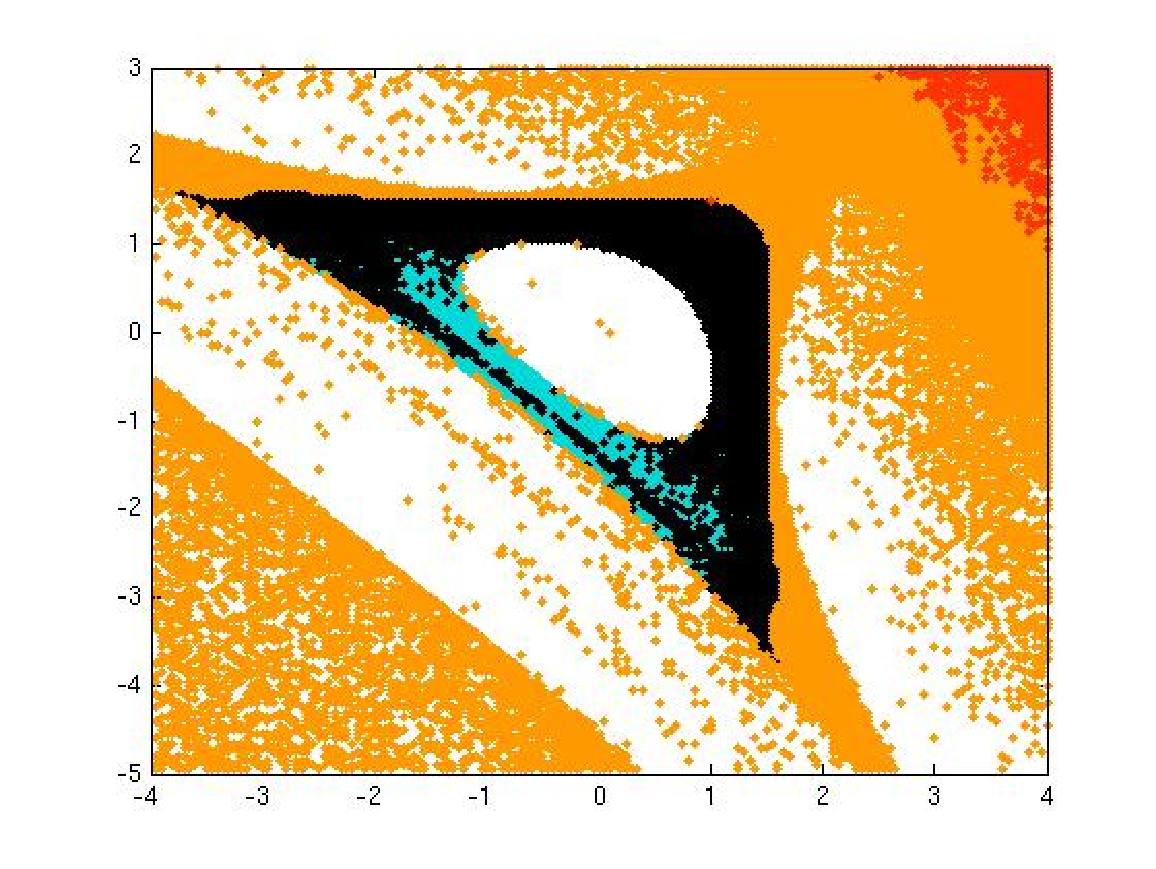}
	\caption{The amoeba of $f := Z_1^2 Z_2 + Z_1 Z_2^2 + c \cdot Z_1 Z_2 + 1$ approximated with {\sc SOStools} for $c = 2$ and $c = -4$. See the text of Example \ref{Exa:GenusOne} for an explanation of the colors and an adress of the numerical issues.}
	\label{FigureExampleSOSAmoeba}
\end{figure}
\fi

\subsection*{The diameter of inner complement components}

\iffullversion
We briefly discuss
\fi 
that the SOS-based certificates can also be used for more
sophisticated questions rather than the pure membership problem.
For this, we consider a family of polynomials $f \in \C[Z_1, \ldots, Z_n]$
whose Newton polytope is a simplex and which have $n+2$ monomials such that one 
of them is located in the interior of the simplex. Amoebas of polynomials of
this class have at 
most one inner complement component (see \cite{theobald-wolff-2011} for a 
comprehensive investigation of that class). 

Let $f := \sum_{i = 0}^{n+1} b_i \cdot Z^{\alpha(i)}$, and 
let $b_0$ denote the coefficient of the inner monomial. 
By the results in \cite{theobald-wolff-2011}, for $|b_0| \to \infty$ the inner 
complement component appears at the 
image under the $\Log$--map of a minimal point $\delta$ of the function 
$\hat{f} := \left| \frac{f}{\arg(b_0) \cdot Z^{\alpha(0)}} \right|$. 
$\delta$ is explicitly computable, and $|\delta|$ is unique. 
This allows to certify that a 
complement component of the unlog amoeba has a certain diameter $d$ under the scaling $|Z| \mapsto |Z|^2$ of the (unlog) amoeba basis space by solving the SDP corresponding to
\[
	\sum_{j = 1}^3 s_jg_j + H + 1 \ = \ 0
\]
with $s_1 = \sum_{i = 1}^n (|\delta_i|^2 - |Z_i|^2)^2 - d^2 / 4$,
$s_2 = \sum_{i = 0}^{n+1} \left(b_i \cdot Z^{\alpha(i)}\right)^{\re}$ and $s_3 = \sum_{i = 0}^{n+1} \left(b_i \cdot Z^{\alpha(i)}\right)^{\im}$,
where $g_j \in \C[Z]$ (restricted to some total degree) and 
$H$ is an SOS polynomial.

Feasibility of the SDP certifies that there exists no point 
$v \in \V(f) \cap \partial \mathcal{B}_{d/2}(\delta)$ (where
$\mathcal{B}_{d/2}(\delta)$ denotes the ball with radius $d/2$ centered in $\delta$) in the rescaled amoeba basis space.
Hence, the corresponding inner complement component of the 
unlog amoeba has at least a diameter $d$ in that space. We have to investigate the rescaled basis space of the unlog amoeba in order to transform the generic condition $(|\delta_i| - |Z_i|)^2 - d^2 / 4$ on the basis space of $\mathcal{U}_f$ into a polynomial condition, which is given by $s_1$ here.

Note that this works not only for polynomials in the class under investigation,
but for every polynomial as long as one knows, where a complement component appears.

\begin{example}
As before, let $f := Z_1^2Z_2 + Z_1Z_2^2 + c \cdot Z_1Z_2 + 1$ with a real parameter $c$. For
this class, the inner complement component appears at the point $(1,1)$ 
and thus under the $\Log$-map at the origin of $\Log(\R^2)$. 
The inner complement component exists for $c > 1$ 
and $c < -3$ (cf.\ \cite{theobald-wolff-2011}). 
We compute a bound for the diameter of the inner complement component using the upper SDP for the intervals $[-3,-9]$ and $[1,7]$ with steplength $0.1$. For any of these points we compute 14 SDPs in
order to estimate the radius (based on binary search).
In the rescaled amoeba basis space we
obtain the bounds shown in Figure~\ref{FigureDiameter}.

\begin{figure}[t]
	\begin{center}
	\begin{picture}(400,200)(0,0)
		\put(0,0){\includegraphics[scale=0.45]{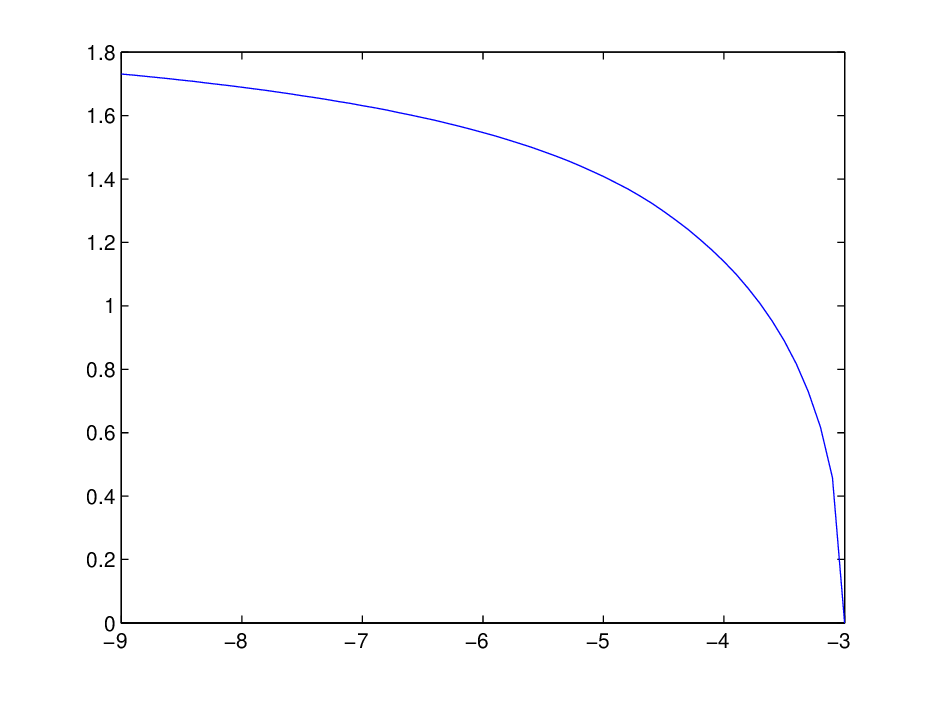}}
		\put(200,0){\includegraphics[scale=0.45]{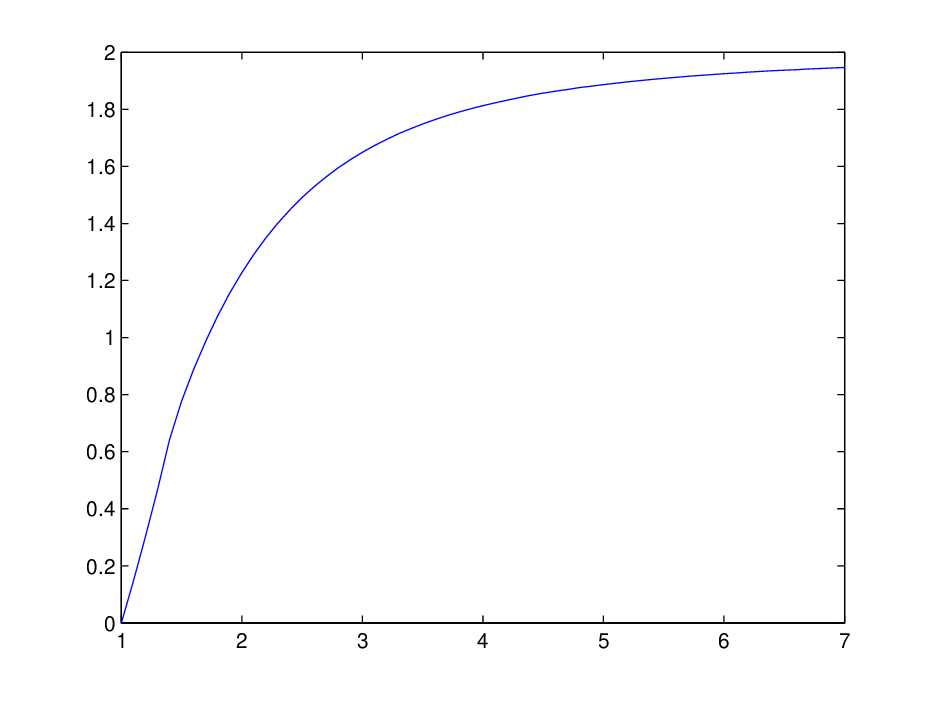}}
	\end{picture}
	\caption{Lower bound for the diameter of the inner complement component.}
	\label{FigureDiameter}
	\end{center}
\end{figure}

Observe that these bounds are lower bounds since feasibility of the SDP certifies membership in the complement of the amoeba but infeasibility only certifies that no certificate with polynomials of degree at most $k$ (i.e., 3 in our case) exists.
\end{example}

This approach also yields lower bounds for the diameter of the inner complement component of the (log) amoeba. The image of the circle $\sum_{i = 1}^n (|\delta_i|^2 - |Z_i|^2)^2 - r^2$ under rescaling and the $\Log$--map,
i.e.\ $\sum_{i = 1}^n (\log|\delta_i| - \log|Z_i|)^2 - r^2$, 
contains the set of points
\begin{eqnarray*}
	\left\{\left(|\delta_1| \cdot e^{r},\ldots,|\delta_n|\right),\ldots,\left(|\delta_1|,\ldots,|\delta_n|\cdot e^{r}\right),\left(|\delta_1|\cdot e^{\frac{r}{\sqrt{n}}},\ldots,|\delta_n|\cdot e^{\frac{r}{\sqrt{n}}}\right)\right\}.
\end{eqnarray*}
By convexity of the complement components in $\mathcal{A}_f$, the simplex
spanned by these points is contained in the inner complement component.
Hence the double radius of the insphere of that simplex is a lower bound for the 
diameter of the inner complement component of~$\mathcal{A}_f$.

\section{Outlook\label{se:outlook}}

We have developed foundations and techniques for approximating
amoebas and coamoebas based on the real Nullstellensatz
and sums-of-squares techniques. While our focus was
on developing the core principles of the computational
methodology, some experimental results were provided to
show the validity of the approach.
Beyond the specific
results we have presented, our approach can be seen
as a first systematic treatment of amoebas
from a (computational) real algebraic point of view
and we think that this viewpoint will have more
potential to offer.

Major current challenges involve both computational
issues (concerning the quality and efficiency of
computations) as well theoretical questions on
on the real algebraic viewpoint on amoebas. To
name a specific open question in the latter respect,
recall that in Theorem~\ref{th:certificateforlopsided}, 
for a hypersurface amoeba
$\mathcal{A}_f$ we could deduce the order of
a complement component  from the special certificates
we treated in that theorem. However, it is
an open and important structural question
how to deduce the order for a point $w$ in the
complement of $\mathcal{A}_f$ given an arbitrary
Nullstellensatz certificate for $w$.

\subsection*{Acknowledgment.} Thanks to Mihai Putinar for pointing
out reference~\cite{niculescu-putinar-2010} and to an anonymous referee for careful reading and helpful suggestions.
Part of this work was done while the first
author was visiting Universidad de Cantabria in Santander. He would like to
express his gratitude for the hospitality. 

\providecommand{\bysame}{\leavevmode\hbox to3em{\hrulefill}\thinspace}
\providecommand{\MR}{\relax\ifhmode\unskip\space\fi MR }
\providecommand{\MRhref}[2]{%
  \href{http://www.ams.org/mathscinet-getitem?mr=#1}{#2}
}
\providecommand{\href}[2]{#2}

\bibliographystyle{amsplain}
\bibliography{bibapproximatingamoebas}

\end{document}